\documentclass[12pt]{article}
\usepackage{graphicx}
\usepackage{amsfonts}
\usepackage{latexsym}
\usepackage{amsthm}
\usepackage{amsmath}
\usepackage{amssymb,amsmath,amsfonts,graphicx,url,amscd}
\usepackage[T1]{fontenc}
\usepackage{color}
\usepackage{tikz}
\usepackage{oplotsymbl}
\usepackage{authblk}

\theoremstyle{plain}
\newtheorem{thm}{Theorem}[section]

\newtheorem{corol}[thm]{Corollary}
\newtheorem{prop}[thm]{Proposition}
\newtheorem{ques}[thm]{Question}

\title{Asymptotics of the number of waves on rational polyhedra}

\author[1]{V.\,L.~Chernyshev}
\author[2,1]{A.~Rukhovich}
\affil[1]{National Research University Higher School of Economics (HSE)}
\affil[2]{Skolkovo Institute of Science and Technology}
\date{}
\begin{document}
\maketitle

\begin{abstract}
  The problem of counting the number of waves arriving at the vertex of a polyhedron is motivated by physics. In the article it was solved for the case of Platonic solid using three nontrivial results from number theory. This growth turns out to be subexponential. Also we prove a subexponential upper bound for all polyhedra with rational total angles at vertices.
\end{abstract}

Let us consider a polyhedron in $3$-space. At the initial moment of time, a concentric wave emerges from the selected vertex $S$ along the surface of the polyhedron with a unit velocity. When it comes to any vertex, a new concentric wave appears. The problem is to understand how many waves will arrive at the selected vertex-recorder $A$ at a given time $t$ (see \cite{Chernyshev} for details of the same dynamical system on a hybrid Riemannian manifolds). Earlier questions that are related to the properties of the Laplace operator were considered (see \cite{Kokotov}, \cite{Lukzen} for details and references). 

Obviously the motion of waves on such a surface is determined by geodesics that emerge from the vertex $S$ and arrive at the vertex $A$.
We have a two-dimensional flat surface with a finite number of conical singularities and we are considering geodesics joining two points.  Such objects were studied in the well-known papers \cite{Eskin}, \cite{Athreya}.
Recent results of \cite{Fuchs}, \cite{Davis}, \cite{Balaji} allow us to solve this problem for regular tetrahedron, octahedron, icosahedron and cube. The number of waves in the vertex is determined by the number of broken geodesic lines, the length of which does not exceed the given parameter. If we know the distribution of components of these broken lines, then using the results from analytical number theory (theorems on the distribution of abstract primes, see \cite{Knopfmacher}, \cite{Nazaikinskii} for details), we can find the growth of the number of waves. 

A close setting was recently studied by Colognese and Pollicott~\cite{Pollicott}.
They considered a circle of radius $t$ on a translation surface -- the set of points on the universal cover of a surface at the distance $t$ from a given point.
Obviously, points on a circle are connected to the initial point either by a geodesic of length $t$ or by broken line of geodesics connecting the vertices, such that the total length of the broken line is $t$.
They show that the length of circle of radius $t$ in exponential in $t$.
The value in this exponent is so-called volume entropy of the surface (see~\cite{Pollicott_Vol},~\cite{Manning}).
These circles are closely related to front of the wave, actually, the projection from the universal cover take points of a circle of radius $t$ to points on the front of the wave at time $t$.
This projection glues together some intervals of the circle: while the length of circle grows exponentially, our result (Corollary~\ref{RT_rat}) implies that the growth of the length of the wavefront is subexponential.
Colognese and Pollicott showed in~\cite{Pollicott} that the uniform measure on circle of radius $t$ converges to the volume measure on a surface as $t\to\infty$.
It is yet unknown, whether same is true for uniform measure on a wavefront.

\section{Geodesics on square-tiled and triangle-tiled surfaces} \label{sec:geodcube}

We start with the case of surfaces tiled of unit squares or regular triangles.
Cube, regular tetrahedron, octahedron and icosahedron belong to this class.
It is well known that the geodesics between the vertices of cube or regular tetrahedron correspond to segments between lattice points in the plane for square or triangular lattice.

Let us call a point $p$ of square (respectively triangular) lattice \textit{irreducible} if the segment from the origin to $p$ does not contain any other lattice point.
For square lattice the equivalent condition is that the coordinates of $p$ are coprime integers.

\begin{prop}
Let us take some vertex $v$ of the tetrahedron.
Then all geodesics from $v$ to the vertices of the tetrahedron (including $v$ itself) are in one-to-one correspondence with the set of irreducible points of triangular lattice inside a sector of angle $\frac{4\pi}{3}$.

Let us take some vertex $v$ of the cube.
Then all geodesics from $v$ to the vertices of the cube (including $v$ itself) are in one-to-one correspondence with the set of irreducible points of square lattice inside a sector of angle $\frac{3\pi}{2}$.
\end{prop}

Thus we know the asymptotics as $l\to\infty$ of the number of geodesics from one vertex to others with length less than $l$ -- it coincides with the asympotics of irreducible lattice points in the sector of fixed angle inside the circle of raduis $l$, which is $\frac{3\pi}{4}l^2$ for cube and $\frac{2\pi}{3}l^2$ for tetrahedron.
The next natural question here is to understand how much of these geodesics connect $v$ to each of the vertices individually rather than all together.
This question is answered by \cite{Fuchs} for tetrahedron and cube and by \cite{Davis} for all platonic solids.
Is is shown that each of the vertices gets asymptotically a fixed fraction of all the geodesics from $v$.

But we need to find asymptotics of the number of lengths of geodesics, not of the number of geodesics.
These can be different since different geodesics may have equal length, e.g. points with coordinates $(11,3)$ and $(9, 7)$ on a square lattice have equal distance $\sqrt{130}$ from the origin.
Therefore, our question is to find the set of all possible norms of irreducible lattice vectors.
Here again arises a question how much of these norms correspond to geodesics from $v$ to some individual vertex, but we do not know the answer.
We will only be interested in the asymptotics of these norms for geodesics to any vertex of our polyhedron.
Such norms are exactly the numbers of the form $\sqrt{x^2+y^2}$ for square lattice and of the form $\sqrt{x^2+xy+y^2}$ for triangular one with integer $x, y$.
Denote these sets of numbers by $A_\triangle$ and $A_\square$.

\begin{prop}
The set of lengths of geodesics between vertices of cube is exactly the set $A_\square$ of lengths of vectors from square lattice.
The set of lengths of geodesics between vertices of tetrahedron (or octahedron, or icosahedron) is exactly the set $A_\triangle$ of lengths of vectors from triangular lattice.
\end{prop}

The asymptotics of such numbers is credited to Landau \cite{Landau} and Ramanujan \cite{Ramanujan}.
From their results (see \cite{Balaji}) follows that

\begin{prop}
The number $a_\square(l) = \#(A_\square \cap (0, l))$ of different lengths of vectors from square lattice not exceeding $l$ grows as $\frac{\gamma_\square l^2}{2\sqrt{\ln l}}$ where $\gamma_\square$ is called the Landau---Ramanujan constant and equals approximately $0.76422$.

Analogously for triangular lattice: $a_\triangle(l) \simeq \frac{\gamma_\triangle l^2}{2\sqrt{\ln l}}$ with $\gamma_\triangle$ approximately $0.64$.

The constants $\gamma_\square$ and $\gamma_\triangle$ are defined by 
$$\gamma_\square = \left(\frac{1}{2}\prod_{p\equiv 3(mod\ 4)}\frac{p^2}{p^2-1}\right)^{1/2},\quad \gamma_\triangle = \left(\frac{1}{2\sqrt{3}}\prod_{p\equiv 2(mod\ 3)}\frac{p^2}{p^2-1}\right)^{1/2},$$
where product is over the set of primes.
\end{prop}

Then the lengths of broken geodesics with vertices at the vertices of regular tetrahedron/octahedron/icosahedron (respectively cube) are exactly the sums of the numbers from $A_\triangle$ (respectively $A_\square$) with non-negative integer coefficients.
Note that different such sums may have equal values, even sum numbers from $A_\triangle$ and $A_\square$ may be integer multiples of others.
Let us drop these values and consider $B_\triangle$ and $B_\square$ to be the sets of values from $A_\triangle$ and $A_\square$ which are not positive integer multiples of other values.
It is clear that squares of values from $B_\triangle$ and $B_\square$ are exactly square-free positive integers of the form $x^2+xy+x^2$ and $x^2+y^2$ respectively.

The asymptotics of these numbers can be shown \cite{Cohen} to be $\frac{\pi^2}{6}$ times less than the previous asymptotics:

\begin{prop}
The number $b_\square(l) = \#(B_\square \cap (0, l))$ grows as $\frac{3\gamma_\square l^2}{\pi^2\sqrt{\ln l}}$.

Analogously for tetrahedron: $b_\triangle(l) \simeq \frac{3\gamma_\triangle l^2}{\pi^2\sqrt{\ln l}}$.
\end{prop}

The lengths of broken geodesics with vertices at the vertices of cube (respectively tetrahedron, or octahedron, or icosahedron) are still the sums of the numbers from $B_\square$ (resp. $B_\triangle$) with non-negative integer coefficients.
Any two such sums are different since the set $\{\sqrt{n}\ |\ n \text{ is a square-free positive integer}\}$ is linearly independent over $\mathbb{Q}$.

\section{Asympotics for the number of waves}
\subsection{Abstract prime number theorem}

We will use the abstract additive prime number theorem. Let us consider an arithmetical semigroup $G$ generated by products of the sequence $p_1, p_2,...$ of distinct elements of a semigroup (called abstract primes). If we have a norm on a semigroup then we could introduce two functions:  counting function of the elements of a semigroup (i.e. abstract integers) $N_G(x)$ that are smaller than $x$ and a counting function of abstract primes $\pi_G(x)$. Now we could investigate how assumptions about the asymptotic behaviour of one of these functions as $x \rightarrow \infty$ influences that of the other. In the general setting, any theorem of this kind could be called an ``abstract prime number theorem'' (see \cite{Knopfmacher}).

John Knopfmacher (see \cite{Knopfmacher}) wrote: ``The abstract additive prime number theorem is closely related to classical theorems of additive analytic number theory due to Hardy and Ramanujan, and concerning arithmetical functions of the same general type as the classical partition function $p(n)$. In fact, the theorem itself may to a large extent be regarded as a reformulation of a certain well-known `Tauberian' theorem of Hardy and Ramanujan''.

\begin{thm}[Additive Abstract Prime Number Theorem] 
\label{AAPNT}
Let $G$ denote an additive arithmetical semigroup such that there exist positive constants $C$ and $\kappa$, and a real constant $\nu$, such that
$$\pi_G^{\#}(x) \sim Cx^{\kappa}(\ln x)^{\nu} as \ x \rightarrow \infty.$$
Then, as $x \rightarrow \infty$,
$$N_G^{\#} = exp\{[c_G + o(1)]x^{\frac{\kappa}{(\kappa + 1)}}(\ln x)^{\frac{\nu}{(\kappa + 1)}}\},$$
where
$$c_G=\kappa^{-1}(\kappa+1)^{\frac{(\kappa - \nu +1)}{(\kappa +1)}}[\kappa C \Gamma (\kappa +1)\zeta (\kappa +1)]^{\frac{1}{(\kappa +1)}}.$$
\end{thm}

\subsection{Waves on lattice surfaces}

From the results of first section we see that for regular tetrahedron and cube we have $\kappa=2$ and $\nu=-1/2$.

Let us substitute this to the theorem~\ref{AAPNT}. We obtain that
$$\pi_G^{\#}(x) \sim Cx^2(\ln x)^{-\frac{1}{2}} \ as \ x \rightarrow \infty.$$
Then, as $x \rightarrow \infty$,
$$N_G^{\#} = exp\{[c_G + o(1)]x^{\frac{2}{3}}(\ln x)^{-\frac{1}{6}}\},$$
where
$$c_G=\frac{3^{\frac{7}{6}}}{2}[2 C \Gamma (3)\zeta (3)]^{\frac{1}{3}}.$$

Here $\zeta (3)\simeq 1,202 056 903 159 594 285 399 738 $ is 
the Apery's constant and $\Gamma (3)=2!=2.$

So we have $$c_G=\frac{3^{\frac{7}{6}}}{2}[4 C \zeta (3)]^{\frac{1}{3}}.$$
Now we are ready to formulate the theorem for cube and regular tetrahedron.

\begin{thm}[Asymptotics of Waves in Vertices for tetrahedron, octahedron, icosahedron and cube]
\label{RT_Cube}
The asymptotics of the number of waves that arrive up to the time $t$ at all vertices of
\begin{itemize}
    \item regular tetrahedron/octahedron/icosahedron:
$$N_\triangle = exp\left(\left[\frac{\sqrt{27}}{\sqrt[3]{2}}\sqrt[3]{\frac{\zeta(3)\gamma_\triangle}{\pi^2}} + o(1)\right]\frac{t^{\frac{2}{3}}}{\sqrt[6]{(\ln t)}}\right),$$
    \item cube:
$$N_\square= exp\left(\left[\frac{\sqrt{27}}{\sqrt[3]{2}}\sqrt[3]{\frac{\zeta(3)\gamma_\square}{\pi^2}} + o(1)\right]\frac{t^{\frac{2}{3}}}{\sqrt[6]{(\ln t)}}\right).$$
\end{itemize}

Here $\gamma_\square$ is the Landau---Ramanujan constant and $\zeta(s)$ is the Riemann zeta function.
\end{thm}

Note that for the case of a cube we have $c_G\simeq1.8690$ and for the case of a regular tetrahedron we have $c_G\simeq 1.7617$.

We see that the number of arriving waves growth slower than the exponent which perfectly fits the results of computer experiments.

\section{Rational flat surfaces}

Let $D$ be a \textit{rational flat surface}, i.e.\ a flat surface with cone singularities whose total angle is a rational multiple of $2\pi$.
There exist a covering of $D$ ramified over its singular points with such degree that all total angles become integer multiples of $2\pi$.
Such covering $\widehat{D}$ is therefore a \textit{translation surface}.
The sets of length of geodesics between singular points are exactly same for $D$ and $\widetilde{D}$.
For translation surfaces such geodesics are called \textit{saddle connections}.
The asymptotics of the number $N(l)$ of saddle connections with length less than $L$ was studied by Veech~\cite{Veech}, and later by Eskin and Masur~\cite{Eskin}.
They showed that there exist constants $c_1, c_2$ such that for each $l$ holds \[c_1 l^2 <N(l)< c_2 l^2.\]
Moreover for a generic translation surface $N(l)/l^2\to c$ as $l\to \infty$ where $c$ 
does not depend on the surface in any given stratum of the moduli space of translation surfaces.
This number $c$ is called a Siegel---Veech constant.

Since the number of geodesics of length less than $l$ between singular points is the obvious upper bound for the number of \textit{lengths} of such geodesics, we obtain

\begin{prop}
    For any rational flat surface the number $b(l)$ of different lengths of geodesics between singular points grows at most quadratically in $l$.
\end{prop}

From this together with theorem~\ref{AAPNT} we obtain
\begin{corol}[Asymptotics of Waves in the Vertices for Rational Flat Surfaces]\label{RT_rat}
    For any rational flat surface $D$ the asymptotics of the number of waves that arrive up to time $t$ at all vertices
    \[N_D < \exp\left( c t^\frac{2}{3} \right)\text{\  for some c.}\]
\end{corol}

As we show in Section~\ref{sec:geodcube},
the asymptotics can be subquadratic in case of lattice surfaces.
We do not know if this may happen for non-lattice surface, and the most interesting is the case of surfaces glued from regular pentagon.
This case was excessively studied by Tabachnikov, Fuchs, Davis~\cite{DFT}.
They found combinatorial description for set of directions and displacement vectors of geodesics for the case of double pentagon.
Though their results do not imply that for this case $b(l)>c l^2$ since different geodesics can be of same length.

\section{Further work}

\begin{ques}
    Is it true that for a generic rational surface $b(l)\sim l^2$?
\end{ques}

\begin{ques}
    Can $b(l)$ be asymptotically lower than $\frac{c l^2}{\sqrt{\log l}}$ for any rational surface?
\end{ques}

\begin{ques}
    Let us say that rational surface is a \textit{lattice surface} if there is a plane lattice containing all displacement vectors of its geodesics.
    Can $b(l)$ be asymptotically lower than $c l^2$ for any rational non-lattice surface?
    In particular, what is the asymptotics of $b(l)$ for the case of pentagonal surface?
\end{ques}

\section{Acknowledgement} The article was prepared within the framework of the HSE University Basic Research Program. We thank V.\,E. Nazaikinskii, A.\,I. Shafarevich, and V.\,Yu. Protasov for useful discussions. 

\bibliographystyle{plain}

\end{document}